\documentclass[12pt]{amsart}
\usepackage{amssymb}
\usepackage{amsmath}
\usepackage{longtable}
\newcommand\g{{\mathfrak g}}
\newcommand\h{{\mathfrak h}}

\renewcommand{\b}{\mathfrak{b}}

\renewcommand{\a}{\mathfrak{a}}

\newcommand\m{\mathfrak m}

\newcommand\n{\mathfrak n}
\newcommand\q{\mathfrak q}

\renewcommand\k{\mathfrak k}

\renewcommand{\t}{\mathfrak{t}}

\newcommand\Aut{\operatorname{Aut}}

\newcommand\Hom{\operatorname{Hom}}

\newcommand\Gr{\operatorname{Gr}}

\newcommand\K{{\mathbb{K}}}
\newcommand\X{\mathfrak X}
\newcommand\Q{\mathbb Q}

\newcommand\D{\mathcal D}
\newcommand\Z{\mathbb Z}

\newcommand\V{\mathcal V}

\newcommand\Rad{\operatorname{R}}

\newcommand\ord{\operatorname{ord}}
\newcommand\GL{\mathop{\rm GL}\nolimits}

\newcommand\SL{\mathop{\rm SL}\nolimits}

\newcommand\SO{\mathop{\rm SO}\nolimits}

\newcommand\LV{\mathcal{LV}}

\newtheorem{Thm}{Theorem}
\newtheorem{Prop}{Proposition}

\newtheorem{Conj}{Conjecture}

\theoremstyle{definition}
\newtheorem{Ex}[Prop]{Example}
\newtheorem{defi}[Prop]{Definition}
\newtheorem{Rem}[Prop]{Remark}
\newtheorem{Exer}{Exercise}

\textheight=240mm \numberwithin{equation}{section}
\author{Ivan  Losev}
\title{Uniqueness properties for spherical varieties}
\thanks{MIT, Department of Mathematics, 77 Massachusetts Avenue, Cambridge MA 02139, USA}\thanks{ e-mail: ivanlosev@math.mit.edu}
\begin{document}
\begin{abstract}
The goal of these lectures is to explain speaker's results on uniqueness properties of spherical varieties.
By a uniqueness property we mean the following. Consider some special class of spherical varieties. Define
some combinatorial invariants for spherical varieties from  this class. The problem is to determine whether
this set of invariants specifies a spherical variety in this class uniquely (up to an isomorphism). We are interested
in three classes: smooth affine varieties, general affine varieties, and homogeneous spaces.
\end{abstract}
\maketitle
\tableofcontents
\section{Main results}
First of all, let us fix some notation. Throughout the notes, the base field $\K$ is algebraically closed
and of characteristic 0. Let $G$  denote a connected reductive algebraic group,  $B$ a Borel subgroup
in $G$, and  $T$ a maximal torus in $B$. Then the character lattices $\X(T),\X(B)$ of $T$ and $B$, respectively,
are canonically identified. Let $\g,\b,\t$ be the corresponding Lie algebras.

General goal: given some class $C$ of spherical $G$-varieties, establish some combinatorial invariants of a variety
in $C$ such that this variety is uniquely determined by its combinatorial invariants.

Perhaps, a combinatorial invariant, which is the easiest to define, is the {\it weight monoid}.
%that has already appeared in Bravi's lectures.
Let $X$ be a spherical $G$-variety.
Then $\K[X]$ is a multiplicity free $G$-module, that is, the
multiplicity of every irreducible module in $\K[X]$  is at most 1.
By the {\it weight monoid} $\X^+_{G,X}$  of $X$ we mean the set of all highest weights of the $G$-module $\K[X]$.
\begin{Exer}
Using the fact that $\K[X]$ is an integral domain check that $\X^+_{G,X}$ is indeed a submonoid
in $\X(T)$.
\end{Exer}

The following result was conjectured by Knop in the middle of 90's.

\begin{Thm}[\cite{Knop_conj}]\label{Thm:2}  Let $X_1,X_2$ be  {\rm smooth}   affine
spherical varieties with $\X^+_{G,X_1}=\X^+_{G,X_2}$. Then $X_1,X_2$ are $G$-equivariantly isomorphic.
\end{Thm}

Note that if $G=T$, then the claim of Theorem \ref{Thm:1} holds even without the smoothness assumption.
Indeed, in this case $\K[X_i]\cong \K[\X^+_{G,X_i}]$.

However, for $G\neq T$ the theorem no longer holds if we omit the smoothness condition.  Indeed, consider the
tautological $G:=\SO(3)$-module $\K^3$. Let $q$ be a $G$-invariant  non-degenerate quadratic form
on $\K^3$. Consider $q$ as a map $\K^3\rightarrow \K$ and let $X_0,X_1$ be the fibers of 0 and 1,
respectively.
\begin{Exer}
Show that $\K[X_0]\cong \K[X_1]$ as $G$-modules and, more precisely, that any $G$-module occurs both
in $\K[X_0],\K[X_1]$ with multiplicity 1.
\end{Exer}
But, of course, $X_0,X_1$ are not isomorphic
as algebraic varieties, for $X_1$ is smooth, but $X_0$ is not.

To remedy the situation one needs to consider a more subtle invariant of spherical varieties:
the valuation cone, see \cite{Knop5}, Corollary 1.8 and Lemma 5.1, or \cite{Timashev},
Section 15 and 21. We denote the valuation cone of a spherical
$G$-variety $X$ by $\V_{G,X}$.

\begin{Thm}[\cite{Knop_conj}]\label{Thm:1}
Let $X_1,X_2$ be two affine spherical $G$-varieties such that $\X^+_{G,X_1}=\X^+_{G,X_2}$
and $\V_{G,X_1}=\V_{G,X_2}$. Then $X_1,X_2$ are $G$-equivariantly isomorphic.
\end{Thm}

Finally, and, in a sense, most importantly, there is a uniqueness property for spherical homogeneous
spaces. Here we need three invariants of a spherical $G$-variety $X$.
The simplest one is the {\it weight lattice}
$\X_{G,X}$ consisting of all weights of $B$ in the field $\K(X)$ of rational functions, see \cite{Knop5},
the paragraph after the proof of Theorem 1.7, or \cite{Timashev}, Section 15.  The second invariant is the set of $B$-stable prime divisors of $X$ denoted by
$\D_{G,X}$. This is a finite set and we equip it with two maps. The first one maps $\D_{G,X}$ to
$\X_{G,X}^*:=\Hom_{\Z}(\X_{G,X},\Z)$, see \cite{Knop5}, Section 2, page 8, or \cite{Timashev},
Section 15 (in both these papers the map is denoted by $\rho$). We denote the image of $D$ by $\varphi_D$. The second map we need maps $D\in \D_{G,X}$ to its
(set-wise) stabilizer $G_D\subset G$. By definition, $G_D$ is a parabolic subgroup in $G$ containing $B$.
Finally, the last invariant we need is $\V_{G,X}$.

The following theorem was, essentially, conjectured by Luna, \cite{Luna4}.

\begin{Thm}[\cite{unique}]\label{Thm:3}
Let $X_1,X_2$ be spherical homogeneous spaces such that $\X_{G,X_1}=\X_{G,X_2}, \D_{G,X_1}=\D_{G,X_2},\V_{G,X_1}=\V_{G,X_2}$. Then $X_1,X_2$ are $G$-equivariantly isomorphic.
\end{Thm}
The equality $\D_{G,X_1}=\D_{G,X_2}$ requires some explanation (the other two are just equalities
of subsets in some ambient set). This equality means that there is a bijection $\iota:\D_{G,X_1}\rightarrow
\D_{G,X_2}$ with $\varphi_{\iota(D)}=\varphi_D$ and $G_{\iota(D)}=G_D$. However, let us note that even if such a bijection exists it is, in general, not unique. Indeed, any $G$-equivariant automorphism $\varphi$ of $X_2$
induces the bijection $\varphi_*:\D_{G,X_2}\rightarrow\D_{G,X_2}$ intertwining the two maps.
So we can compose $\iota$ with $\varphi_*$. It turns out (this is a non-trivial result) that any two bijections
$\iota$ differ by some $\varphi_*$.

\begin{Rem}
It is often more convenient to deal with the system of spherical roots of $X$.
It can be defined as follows. It is known from the work of Brion, \cite{Brion}, see also \cite{Timashev},
Section 22, that $\V_{G,X}$ is a Weyl chamber for the action of a finite reflection group
$W_{G,X}$ (the Weyl group of $X$) on $\Hom_\Z(\X_{G,X},\Q)$. So we can take linearly independent
primitive elements $\alpha_1,\ldots,\alpha_k\in \X_{G,X}$ such that $\V_{G,X}$ is given by the inequalities
$\alpha_i\leqslant 0$. The set $\{\alpha_1,\ldots,\alpha_k\}$ is called the system of spherical
roots of $X$,
we will denote it by $\Psi_{G,X}$. If $\X_{G,X}$ is specified, then $\V_{G,X}$ can be recovered
from $\Psi_{G,X}$ and vice versa.
\end{Rem}
%\section{Toolkit}

In what follows I will call $\X_{G,X},\D_{G,X},\Psi_{G,X}$ the basic combinatorial invariants
of the spherical $G$-variety $X$.

To finish the section let us consider an application of Theorem \ref{Thm:2}, which, in fact,
motivated Knop to make his conjecture. This application is the Delzant conjecture from the theory
of Hamiltonian actions of compact groups, a reader is referred to \cite{GS} for definitions.

Let $K$ be a connected compact Lie group and $\k$ be the Lie algebra of $K$. Fix a maximal torus $T_K\subset K$
and let $\t_K$ denote the corresponding Lie algebra. Fix a Weyl chamber $C\subset \t_K$.
Our goal is, again, to present combinatorial invariants separating {\it multiplicity free} Hamiltonian $K$-manifolds. Recall that one of the equivalent definitions of a multiplicity free compact Hamiltonian
$K$-manifold $M$ is that a general orbit of $K$ on $M$ is a coisotropic submanifold.

Let $M$ denote a multiplicity free compact Hamiltonian manifold with moment map $\mu:M\rightarrow \k$.
Recall the moment polytope $\Delta(M)=\mu(M)\cap C$. This is the first invariant we need.
The second one is the so called {\it principal isotropy group}, which will be denoted by
$K(M)$. It is defined as the stabilizer
of a general point $x\in \mu^{-1}(C)$. It turns out that this stabilizer does not depend
on the choice of $x$.

\begin{Conj}[Delzant]
Let $M_1,M_2$ be multiplicity free compact Hamiltonian $K$-manifolds such that
$\Delta(M_1)=\Delta(M_2)$ and $K(M_1)=K(M_2)$. Then $M_1,M_2$ are $K$-equivariantly
symplectomorphic.
\end{Conj}

Knop derived this conjecture from his own in mid 90's, however the proof was never published.
In a sentence, the relation between the two conjectures is that Knop's is a local version of Delzant's.

\section{Sketch of  reduction of the affine case to the homogeneous case}
The proof of all three main theorems is based on  inductive arguments. We basically have two kinds
of inductive steps. One works for homogeneous spaces only and is based on Knop's theory
of inclusions of spherical subgroups, \cite{Knop5}, Section 4, the other works for all varieties and is based
on the Brion-Luna-Vust local structure theorem. We will explain a variant of this theorem
due to Knop, \cite{Knop3}.

\begin{Thm}\label{Thm:loc_struc}
Let $X$ be some normal $G$-variety and $\widetilde{D}$ be a $B$-stable effective Cartier divisor.
Let $P$ be the stabilizer of $\widetilde{D}$ and $X^0$ denote the complement of $\widetilde{D}$ in $X$.
Finally, let $M$ be the Levi subgroup of $P$ containing $T$ so there is the Levi
decomposition $P=M\rightthreetimes \Rad_u(P)$. Then there is an $M$-stable
subvariety $\Sigma\subset X^0$ such that the natural morphism
$\Rad_u(P)\times \Sigma\rightarrow X^0, (p,s)\mapsto ps,$ is an isomorphism.
\end{Thm}

It is easy to see that $\Sigma$ is $M$-spherical provided $X$ is $G$-spherical.
Also one can check that $\Sigma$ is affine provided $X$ is. The latter follows from the general fact that
a complement to a divisor in an affine variety is affine provided the divisor is Cartier.

Actually, one can recover combinatorial invariants of $\Sigma$ from those of
$X$:
\begin{itemize}
\item[(A)] $\X_{M,\Sigma}=\X_{G,X}$.
\item[(B)] As an abstract set, $\D_{M,\Sigma}=\D_{G,X}\setminus \D$, where $\D$ is the set of irreducible
components of $\widetilde{D}$. For $D\in \D_{M,\Sigma}$ we have  $M_D=M\cap G_D$ and the vector
$\varphi_D$ is the same as before.
\item[(C)] $\Psi_{M,\Sigma}$ is the intersection of $\Psi_{G,X}$ with the linear span  of
the root system $\Delta(\m)$.
\item[(D)] Suppose $X$ is affine and $D$ is the zero divisor of some $B$-semiinvariant function
$f_{\mu}\in \K[X]$. Then $\X^+_{M,\Sigma}=\X^+_{G,X}+\Z\mu$.
\end{itemize}

\begin{Exer}
Prove (A),(B),(D).
\end{Exer}

Our strategy in the proofs of Theorems \ref{Thm:2},\ref{Thm:1} is to reduce them to Theorem
\ref{Thm:3}. We will concentrate on Theorem \ref{Thm:2} from now on.
We need to check the following two claims:
\begin{itemize}
\item[(*)] Let $X_1,X_2$ be smooth affine spherical varieties such that $\X^+_{G,X_1}=\X^+_{G,X_2}$.
Then $\D_{G,X_1}=\D_{G,X_2}$.
\item[(**)] Let $X_1,X_2$ be as in (*), so $\D_{G,X_1}=\D_{G,X_2}$. Then $\Psi_{G,X_1}=\Psi_{G,X_2}$.
\end{itemize}

Once (*) and (**) are proved we can deduce Theorem \ref{Thm:2} from Theorem \ref{Thm:3} as follows.

\begin{proof}[Proof of Theorem \ref{Thm:2}]
Let $X_1^0,X_2^0$ be the open $G$-orbits in $X_1,X_2$. One can easily recover the basic combinatorial
invariants of $X_i^0$ from those of $X_i$.
\begin{Exer}
Do it.
\end{Exer}
It follows that $X_1^0,X_2^0$ satisfy the conditions
of Theorem \ref{Thm:3} and so are isomorphic. Identify $\K[X_1^0]\cong \K[X_2^0]$. Then $\K[X_1]=\K[X_2]$
since both are the sums of all $V(\lambda)\subset \K[X_i^0]$ with $\lambda\in \X^+_{G,X_i}$.
\end{proof}

Now we will sketch the proof of (*). An essential ingredient of the proof is the property (D) above.

Namely, choose  noninvertible $\mu\in \X^+_{G,X_i}$. Let $X_1(\mu),X_2(\mu)$ be the $M$-varieties obtained
from $X_1,X_2$ by using the local structure theorem (here $M$ is the stabilizer of $\mu$ in $G$). Then there exists a bijection $\iota_\mu: \D_{M,X_1(\mu)}\rightarrow\D_{M,X_2(\mu)}$ with the required properties.
As we mentioned above, $\D_{M,X_i(\mu)}$ is a subset of $\D_{G,X_i}$. It consists precisely of those
$D\in \D_{G,X_i}$ s.t. $\langle\varphi_D,\mu\rangle=0$. In a sense,  using the local structure
theorem, we can "reveal" all divisors $\D\in \D_{G,X_i}$ such that $\langle\varphi_D,\mu\rangle=0$ for some
noninvertible $\mu\in \X^+_{G,X_i}$. This motivates the following definition.

\begin{defi}
An element $D\in \D_{G,X_i}$ is called {\it hidden} if $\langle\varphi_D,\mu\rangle>0$ for all
$\mu\in \X^+_{G,X_i}\setminus -\X^+_{G,X_i}$.
\end{defi}

As the following example shows, hidden divisor do occur.

\begin{Ex}
Let $G=\SL(n), H=\GL(n-1)\subset \SL(n)$. Then $\D_{G,G/H}$ consists of two elements, and $\X_{G,G/H}$
has rank 1. Both elements of $\D_{G,G/H}$ are hidden.
\end{Ex}

\begin{Exer}
Use an embedding $G/H\hookrightarrow \mathbb{P}(\K^n)\times \mathbb{P}(\K^{n*})$ to prove all claims of the
previous example.
\end{Exer}

There are also other examples, but this one is the most nontrivial, and, in a sense, everything
else reduces to it.

So to construct a bijection $\iota:\D_{G,X_1}\rightarrow \D_{G,X_2}$ we need:
\begin{itemize}
\item To compose different bijections $\iota_\mu$ together (they do not necessarily agree
on intersections). This is relatively easy.
\item To show that one actually has coincidence of stabilizers in $G$ not in different $M$'s.
\item To get reasonable description of all cases with hidden divisors and  deal with them.
\end{itemize}

The last two parts are quite difficult. We are not going to provide
details here.

\section{Sketch of the proof in the homogeneous case}
In this section we will provide a sketch of the proof of Theorem \ref{Thm:1}.

As I mentioned before the proof is based on two types of induction.

{\bf 1. Local structure theorem.} Let me explain how to apply Theorem \ref{Thm:loc_struc} in this case.

Let $H\subset G$ be a spherical subgroup that can be included into a proper parabolic subgroup.
Conjugating, we may assume that $H$ is contained in a parabolic subgroup $Q^-$ that is opposite
to $B$. Let $M$ be the Levi subgroup of $Q^-$ containing $T$ and $Q:=BM$. Consider the projection $\pi:G/H\twoheadrightarrow G/Q^{-}$. Let $\widetilde{D}_0$
be the complement to the open $Q$-orbit in $G/Q^-$. Then $\widetilde{D}_0$ is a divisor,
so we can take $\widetilde{D}=\pi^{-1}(\widetilde{D}_0)$ in Theorem \ref{Thm:loc_struc}.

\begin{Exer}
Show that one can take $Q^-/H$ for the section $\Sigma$.
\end{Exer}

In general, the spherical variety $Q^-/H$ is hard to deal with. However, there are cases when
this variety is affine. Namely, let $H=S\rightthreetimes N$ be a Levi decomposition.
Suppose that $N$ is contained in the unipotent radical $\Rad_u(Q^-)$ of $Q^-$ (this is
always the case when $Q^-$ is a minimal parabolic containing $H$). Then we can conjugate
$H$ by an element from $\Rad_u(Q^-)$ and assume that $S\subset M$. In this case $Q^-/H$
is the homogeneous vector bundle $M*_S (\Rad_u(\q^-)/\n)$ (over $M/S$ with fiber
$\Rad_u(\q^-)/\n$).

{\bf 2. Knop's theory of inclusions of spherical subgroups.}
Let $X:=G/H$ be a homogeneous space.  Then
\begin{enumerate}
\item The set of all subgroups $\widetilde{H}\subset G$ such that $\widetilde{H}\supset H$ and
 $\widetilde{H}/H$ is connected can be described entirely in terms
of  $\X_{G,X},\D_{G,X},\V_{G,X}$. Namely,  subgroups $\widetilde{H}$
are in one-to-one correspondence with pairs $(\a,\D)$ (so called, {\it colored subspaces}),
where $\a$ is a subspace in $\Hom_\Z(\X_{G,X},\Q)$,  $\D$ is a subset in $\D_{G,X}$,
satisfying some combinatorial conditions.
\item Let $\widetilde{H}$ be a subgroup corresponding to a colored subspace
$(\a,\D)$. Then the combinatorial invariants of $G/\widetilde{H}$ can be recovered
from those for $G/H$ and from $(\a,\D)$.
\end{enumerate}

This claim allows to do induction. Namely let us take two homogeneous spaces $X_1=G/H_1,X_2=G/H_2$
satisfying the conditions of Theorem \ref{Thm:2}. Take a minimal subgroup $\widetilde{H}_1$
containing $H_1$ properly. This subgroup gives rise to a colored subspace (for $X_1$).
Take the same colored subspace for $X_2$ and let $\widetilde{H}_2$ be the
corresponding subgroup of $G$ containing $H_2$. But now $G/\widetilde{H}_1$ and $G/\widetilde{H}_2$
have the same basic combinatorial invariants. By inductive assumptions, $\widetilde{H}_1$ and $\widetilde{H}_2$
are conjugate. So we may assume that
\begin{enumerate} \item $\widetilde{H}=\widetilde{H}_1=\widetilde{H}_2$,
\item and the colored
subspaces of the inclusions $H_1\subset \widetilde{H}, H_2\subset \widetilde{H}$ are the same.
\end{enumerate}
There is a subtlety in the second part coming from the fact that a bijection between the
sets of divisors is non-unique, but this can be fixed.

We described the induction steps but did not mention the base. Well, the base is the case when
neither $H_1$ nor $H_2$ can be included into a proper parabolic. But then $H_1,H_2$
are both reductive and one can use the classification due to Kr\"{a}mer, Brion and Mikityuk, \cite{Kramer},
\cite{Brion_class}, \cite{Mikityuk},
to prove Theorem \ref{Thm:3}.

Now we are ready to give a sketch of the proof of Theorem \ref{Thm:3}.

{\it Step 1.} We may assume that $N_G(H_i)^\circ\subset H_i$ for $i=1,2$. This can
be deduced from Luna's results, \cite{Luna5}.

{\it Step 2.} Let $H_1=S\rightthreetimes N_1$ be a Levi decomposition. Then there exists a subgroup
$\widetilde{H}$ containing $H_1$ and such that $\widetilde{H}=S\rightthreetimes \widetilde{N}$, where
$\widetilde{\n}/\n_1$ is an irreducible $S_1$-module. As explained above, we may assume that $H_2$
is contained in $\widetilde{H}$. Thus $\dim H_2/\Rad_u(H_2)\leqslant \dim \widetilde{H}/\Rad_u(\widetilde{H})=\dim S$. Now from the symmetry between $H_2$ and $H_1$ we see that the previous inequality turns into equality
so we can write $H_2=S\rightthreetimes N_2$. Since $\widetilde{H}$ is a minimal subgroup containing
$H_2$ we get that $\widetilde{\n}/\n_2$ is an irreducible $S$-module.

{\it Step 3.}
Pick a minimal parabolic subgroup
$Q^-$ containing $N_G(\widetilde{H})$. Note that $N_1,N_2\subset \widetilde{N}\subset \Rad_u(Q^-)$.
We may assume that $Q^-$ is opposite to $B$
and that $S\subset M$. From the local structure theorem we deduce that the affine
spherical $M$-varieties $Q/N_1=M*_S(\Rad_u(\q^-)/\n_1),Q/N_2=M*_S(\Rad_u(\q^-)/\n_2)$ have the same basic combinatorial invariants. Therefore, using the induction assumption, we obtain that there is an $M$-equivariant
isomorphism $M*_S(\Rad_u(\q^-)/\n_1)\rightarrow M*_S(\Rad_u(\q^-)/\n_2)$.
It follows that the $S$-modules $\Rad_u(\q^-)/\n_1,\Rad_u(\q^-)/\n_2$ are
conjugate under the action of $N_M(S)$.

{\it Step 4.} Suppose for a moment that $\Rad_u(\q^-)/\n_1$ and $\Rad_u(\q^-)/\n_2$ are
actually isomorphic as $S$-modules. Then $\widetilde{\n}/\n_1,\widetilde{\n}/\n_2$ are isomorphic
as $S$-modules. The following exercise implies that $[\widetilde{\n},\widetilde{\n}]\subset\n_1\cap \n_2$.

\begin{Exer}
Let $\widetilde{\n}$ be a nilpotent Lie algebra, $S$ an algebraic group acting on $\widetilde{\n}$
by Lie algebra automorphisms. Let $\n$ be an $S$-stable subalgebra in $\widetilde{\n}$
such that $\widetilde{\n}/\n$ is an irreducible $S$-module. Then $[\widetilde{\n}/\widetilde{\n}]\subset \n$.
\end{Exer}

Now since $[\widetilde{\n},\widetilde{\n}]\subset \n_1\cap\n_2$ and $\widetilde{\n}/\n_1,\widetilde{\n}/\n_2$ are
$S$-equivariantly isomorphic, there is a family $\n(t), t\in \mathbb{P}^1,$ of $S$-stable subspaces of
$\widetilde{\n}$ containing $\n_1\cap\n_2$ and such that $\n(0)=\n_1,\n(\infty)=\n_2$.
So $\h_1$ can be deformed to $\h_2$. Now the rigidity
results of Alexeev and Brion, \cite{AB}, together with the assumption of step 1 imply that
$\h_1,\h_2$ are $G$-conjugate. Using the equality $\X_{G,X_1}=\X_{G,X_2}$ one can check that $H_1,H_2$
are also conjugate.

{\it Step 5.} Recall that we still have not proved that the $S$-modules $\Rad_u(\q^-)/\n_1,\Rad_u(\q^-)/\n_2$
are actually isomorphic. We only checked that they are conjugate under the action of $N_M(S)$ on the
 set of $S$-modules. So pick an element $\gamma\in N_M(S)$ that conjugates $\Rad_u(\q^-)/\n_2$ to
$\Rad_u(\q^-)/\n_1$. One can show that $\gamma^2$ acts trivially on the set of modules. It follows that $\gamma$ fixes $\Rad_u(\q^-)/\widetilde{\n}$.
So $\gamma$ can be lifted to an automorphism of $M*_S(\Rad_u(\q^-)/\widetilde{\n})$. Now the crucial observation
is that $\gamma$ can be actually lifted to an element  $g\in N_G(\widetilde{H})$. This follows from the description
of the group of equivariant automorphisms of a spherical variety. This description will be discussed
briefly in the next section. Replacing $H_2$ with $gH_2g^{-1}$ we obtain
$\Rad_u(\q^-)/\n_1\cong \Rad_u(\q^-)/\n_2$.

This completes the proof.

\section{Equivariant automorphisms and Demazure embeddings}
We have already mentioned that the group of equivariant automorphisms $\Aut^G(X)$ of a spherical
$G$-variety $X$ can be recovered (at least, when $X$ is affine or homogeneous) from the basic
combinatorial invariants. Let us state this description here.

At the first glance, the group $\Aut^G(X)$ has nothing to do with $\X_{G,X},\D_{G,X},\Psi_{G,X}$.
However,  Knop, \cite{Knop8}, discovered a relation. Let us recall that to any $\lambda\in \X_{G,X}$
we have assigned a unique (up to rescaling) rational function $f_\lambda\in \K(X)$.
Pick $\varphi:\Aut^G(X)$. Then $\varphi.f_\lambda$ is again $B$-semiinvariant of weight
$\lambda$. So there is a unique scalar $a_{\varphi,\lambda}$ such that $\varphi.f_{\lambda}=a_{\varphi,\lambda}f_{\lambda}$.

\begin{Exer}
The map $a_{\varphi}:\X_{G,X}\rightarrow \K^\times,a_\varphi(\lambda)=a_{\varphi,\lambda}$
is a character of $\X_{G,X}$.
\end{Exer}

Define the {\it root lattice} $\Lambda_{G,X}=\bigcap_{\varphi\in \Aut^G(X)}\ker a_{\varphi}$.
The following exercise explains the terminology.

\begin{Exer}
Let $X:=H$ be a connected reductive  algebraic group, and let $G:=H\times H$ act on $H$ by two-sided multiplications.
Then $\X_{G,X}$ is identified with the weight lattice of $H$. Show that $\Lambda_{G,X}$ is the root
lattice of $H$.
\end{Exer}

Similarly, one can define the homomorphisms $a_{\lambda}:\Aut^G(X)\rightarrow \K^\times$ for $\lambda\in \X_{G,X}$.
As Knop proved in \cite{Knop8}, $\bigcap_{\lambda\in \X_{G,X}}\ker a_\lambda=\{1\}$. So the group
$\Aut^G(X)$ is identified with $(\X_{G,X}/\Lambda_{G,X})^*$. Therefore to describe $\Aut^G(X)$
it will be enough to describe the subgroup $\Lambda_{G,X}\subset \X_{G,X}$.

We will do slightly better: we will describe a distinguished basis in $\Lambda_{G,X}$. Namely,
one can construct a set of vectors $\overline{\Psi}_{G,X}\subset \Lambda_{G,X}$ completely
analogously to the construction $\Psi_{G,X}\subset \X_{G,X}$.  Knop proved in
\cite{Knop8} that $\overline{\Psi}_{G,X}$ is a basis in $\Lambda_{G,X}$.

Clearly, any element in $\overline{\Psi}_{G,X}$ is proportional to exactly one element
on $\Psi_{G,X}$. In fact, all coefficients are either 1 or 2. Theorem 2 from \cite{unique}
explicitly explains how we distinguish between 1 and 2 analyzing $\Psi_{G,X}$
and $\D_{G,X}$. The description is technical and therefore we omit it (however, see Conjecture \ref{Conj:3}
in the next section; in the spherical case it boils down to the description).

Instead we will explain a result, which motivated us to  state and prove Theorem 2 from
\cite{unique} (applications to the proof of Theorem \ref{Thm:3} were discovered later).
This result is the conjecture of Brion, \cite{Brion}, that Demazure embeddings are smooth.

Let us explain what "Demazure embedding" means. Take a self-normalizing spherical subalgebra
$\h\subset\g$. Then we can consider the $G$-orbit $G\h$, where $\h$ is viewed
as a point in the Grassmann variety $\Gr(\g)$. Take the closure $\overline{G\h}$.
Brion conjectured that the closure is smooth.

The first indication that the Brion conjecture can be related to the description of
$\Aut^G(X)$ is as follows. Let us note that the open $G$-orbit in $\overline{G\h}$ is nothing
else but $G/N_G(\h)$. It is easy to notice that this homogeneous space has no nontrivial
$G$-equivariant automorphisms.

\begin{Exer}
Show it.
\end{Exer}

In \cite{Knop8} Knop proved that the {\it normalization} of the Demazure embedding is smooth. 
Brion's conjecture was proved by Luna for type A in \cite{Luna_Dem}. Using the description
of $\Aut^G(X)$ the author was able to extend Luna techniques to the general case, see
\cite{Demazure}.

\section{Generalizations}
All three theorems, as well as the description of the group of equivariant
automorphisms,  have interesting conjectural generalizations to the general (not necessarily
spherical) case. The invariants appearing in these generalizations are mostly quite difficult to define
 and are even more difficult to deal with.

Again, the conjecture that is easiest to state deals with smooth affine varieties.
For a generalization of a weight monoid $\X^+_{G,X}$ we take the algebra $\K[X]^U$ of $U$-invariants
equipped with the natural $T$-action. Here $U$ is the maximal unipotent subgroup of $B$.
The algebras of $U$-invariants were extensively studied in Invariant theory.
\begin{Exer}
Check that if $X$ if an affine spherical variety, then $\K[X]^U$ is $T$-equivariantly isomorphic
to the monoid algebra $\K[\X^+_{G,X}]$.
\end{Exer}

\begin{Conj}\label{Conj:1}
If $X_1,X_2$ are smooth affine $G$-varieties such that the algebras $\K[X_1]^U,\K[X_2]^U$ are
$T$-equivariantly isomorphic, then $X_1,X_2$ are $G$-equivariantly isomorphic.
\end{Conj}

In general, it is known, see \cite{AB}, Corollary 2.9, that there are only finitely many isomorphism
classes of affine $G$-varieties $X$ with a fixed $\K[X]^U$. 

Unfortunately, it seems to be unlikely that one will  be able to prove this conjecture without proving an analog
of Theorem \ref{Thm:3}. This analog should deal with birational classification
of $G$-varieties instead of just the classification of homogeneous spaces. Therefore let us
describe the invariants which (conjecturally) should be used. Below $X$ stands for a normal
irreducible $G$-variety.

For a generalization of the weight lattice $\X_{G,X}$
 we take the set $\K(X)^{(B)}$ of all rational $B$-semiinvariant functions on $X$.
This set is equipped with the multiplication and partially defined
addition, both induced from $\K(X)$.

The other two invariants are now generalized directly. Consider the set $\D_{G,X}$ consisting of all prime Weil
divisors on $X$ that are $B$-stable. Again, there
is a natural map $\D_{G,X}\rightarrow
\Hom_\Z(\K(X)^{(B)\times}/\K^\times,\Z), $ $D\mapsto \varphi_D,
\varphi_D(f):= \ord_D(f)$. Clearly, this set is not a birational invariant:
to get one just consider the set $\D^0_{G,X}=\{D\in \D_{G,X}| G_D\neq G\}$.

Finally, one can consider the set $\overline{\V}_{G,X}$ consisting of all geometric (i.e., coming
from divisors, in an appropriate sense) $\Q$-valued discrete $G$-invariant
valuations of $\K(X)$. This set is equipped with the map
$\overline{\V}_{G,X}\rightarrow
\Hom_\Z(\K(X)^{(B)\times}/\K^\times,\Q)$ given by
$v\mapsto\varphi_v, \varphi_v(f):=v(f)$. This map is known to be
injective, see \cite{Knop4}, so we consider $\V_{G,X}$ as a subset
in $\Hom_\Z(\K(X)^{(B)\times}/\K^\times,\Q)$.

 We call the triple $(\K(X)^{(B)\times}, \D_{G,X}^0, \overline{\V}_{G,X})$
 the Luna-Vust
(shortly, LV) system of $X$
$X$) and denote it by $\LV_{G,X}$. The reason for this is that these invariants
appeared already in \cite{LV}. See \cite{Timashev}, Sections 12-14, 20-21, for more information
about them.

By an
{\it isomorphism} of two LV systems
$(F_1,\V_1,\D_1),(F_2,\V_2,\D_2)$ we mean a pair $(\psi,\iota)$,
where $\psi:F_1\rightarrow F_2,\iota:\D_1\rightarrow \D_2$ are
isomorphisms satisfying the natural compatibility relations. For
example,  $T$  acts by automorphisms of a LV
system $\LV=(F,\V,\D)$ (the action on $\D$ is trivial). By
$\Aut^{ess}(\LV)$ we denote the quotient of the whole group
$\Aut(\LV)$ by the image of $T$. In a sense, this group consists of "essential" automorphisms
of $\LV$, while automorphisms coming from $T$ are considered as "non-essential".  

We would like to  make the following two conjectures.

\begin{Conj}\label{Conj:2}
Let $X_1,X_2$ be two normal irreducible $G$-varieties with isomorphic LV systems.
Then $X_1,X_2$ are birationally isomorphic (as $G$-varieties).
\end{Conj}

\begin{Conj}\label{Conj:3}
The group of birational $G$-equivariant automorphisms of $X$ surjects onto
$\Aut^{ess}(\LV)$.
\end{Conj}

Here is a very rough strategy that one can use to prove these conjectures. This strategy
is just a direct generalization of the one used in \cite{unique}:

{\it Step 1.} Reduce Conjectures \ref{Conj:2},\ref{Conj:3} to the
case of homogeneous spaces (this step does not occur in the
spherical case).

{\it Step 2.} Prove Conjectures \ref{Conj:2}, \ref{Conj:3} for
affine homogeneous spaces.

{\it Step 3.} Develop the theory of inclusions of subgroups of $G$
on the language of LV systems. Such a theory was developed by Knop,
\cite{Knop5} in the spherical case.

{\it Step 4.} Prove Conjecture \ref{Conj:3} generalizing the proof
of \cite{unique}, Theorem 2.

{\it Step 5.} Prove Conjecture \ref{Conj:2} generalizing the proof
of \cite{unique}, Theorem 1.

Of course, performing (at least some of) these steps is not easy at all.
For example, step 2 should involve some new ideas (this step
in the spherical case relies heavily on the classification).

After Conjecture \ref{Conj:2} is proved one should be able to prove Conjecture
\ref{Conj:1} by verifying analogs of claims (*),(**) from Section 2.

{\bf Acknowledgements.} I would like to thank Michel Brion for his remarks on the previous 
version of this text.

\end{document}